\documentclass[a4paper,12pt]{article}

\usepackage{hyperref}
\usepackage{fullpage}
\usepackage{amsthm,amsmath, amssymb}
\usepackage{palatino,stmaryrd}
\usepackage[usenames, dvipsnames]{color}

\input{xy}
\xyoption{all}
\xyoption{matrix}

\newtheorem{theorem}{Theorem}[section]
\newtheorem{teo}[theorem]{Theorem}

\newtheorem{lem}[theorem]{Lemma}

\newtheorem{fact}[theorem]{Fact}

\theoremstyle{definition}

\newtheorem{defi}[theorem]{Definition}

\DeclareMathOperator{\id}{id}
\DeclareMathOperator{\Id}{Id}

\DeclareMathOperator{\ev}{ev}
\newcommand{\ot}{\otimes}

\newcommand{\wt}{\widetilde}
\newcommand{\wh}{\widehat}

\def\YD{\mathcal{YD}}

\def\II{\mathcal I}

\def\t{\mathfrak t}
\def\u{\mathfrak u}

\def\b{\mathfrak b}
\def\B{\mathfrak B}

\def\ev{\mathrm{ev}}
\def\coev{\mathrm{coev}}

\def\om{\omega}

\def\To{\Rightarrow}

\title{Universal quantum (semi)groups
and Hopf envelopes: Erratum}

\author{Marco Andr\'es Farinati
\thanks{Dpto de Matem\'atica FCEyN UBA - IMAS (Conicet). 
e-mail: mfarinat@dm.uba.ar.
Partially supported by 
UBACyT 2018-2021
``K-teor\'ia y bi\'algebras en \'algebra, geometr\'ia y topolog\'ia''
and
PICT 2018-00858 ``Aspectos algebraicos y anal\'iticos de grupos
 cu\'anticos''.}}

\date{}

\begin{document}
\maketitle

{\em e-mail: mfarinat@dm.uba.ar}

{\em Data availability statement: This manuscript has  no associated data.}

\begin{abstract}
In \cite{F} there is a statement generalizing the results in \cite{FG}.
Unfortunately there is a mistake in a computation that
 affects the main result. I don't know if the main result in \cite{F}
is true or not, but I propose an alternative statement (Theorem
\ref{mainthm}) that was actually
the main motivation in \cite{FG}. This statement answers in an affirmative way
the question whether the localization of the FRT construction with respect
to a quantum determinant is a Hopf algebra, in case the
Nichols algebra associated to the braiding is finite dimensional.
\end{abstract}

\section{Error in "Universal quantum(semi)groups and Hopf envelopes"}

For a bilinear form $b:V\times V\to k$ in a finite dimensional vector space $V$
over a field $k$ with basis
$x_\mu$, defined by
\[
b_{\mu\nu}:=b(x_\mu,x_\nu)
\]
Dubois-Violette and Launer \cite{DVL}
define a Hopf algebra with generators
$t_\lambda^\mu$ ($\lambda,\mu=1,\dots,\dim V$)
and relations (sum over repeated indexes)

\begin{eqnarray}\label{eqdv1}
b_{\mu\nu}t_ \lambda^\mu t_\rho^\nu = b_{\lambda\rho}1
\\
\label{eqdv2}
b^{\mu\nu}t^ \lambda_ \mu t^\rho_ \nu = b^{\lambda\rho}1
\end{eqnarray}

In \cite{F} there is a Lemma 2.1 saying
that equation \eqref{eqdv2} is redundant. Unfortunately the 
proof is incorrect.  I thank
Hongdi Huang and her collaborators
 Padmini Veerapen, Van Nguyen, Charlotte Ure, Kent Washaw and Xingting
 Wang for pointing me up the error. A lot of important consequences in 
\cite{F} are derived form this lemma, mainly Sections 2 and 3:

\begin{itemize}
\item Corollary 2.2 in \cite{F}, saying that $A(b)$, the universal bialgebra associated to a bilinear form, is a Hopf algebra.
\item Theorem 2.4 in \cite{F}, saying that a  universal bialgebra
associated to a specific bilinear form and a quotient of it is a Hopf algebra.
\item And the main result: Theorem 3.10 of \cite{F}, that says that 
$A(c)$ = the FRT construction associated to a solution $c:V\to V\to V\ot V$ 
of the braid equation in $V$ admitting a weakly graded-Frobenius algebra (WGF),
becomes a Hopf algebra when localizing with respect to the quantum determinant associated that WGF algebra.

\end{itemize}

On the other hand, the general universal constructions of Section 1
is independent of Lemma 2.1 and the following parts are still safe:
\begin{itemize}
\item The bialgebraic nature of the construction
(Theorem 1.1)
\item Its universal property (Proposition 1.3).
\item Example of computation 1.4 and Remark 1.5.

\item Section 4: the locally finite graded case and comments on 
other related works.

\end{itemize}

\section{The mistake, and alternatives to Lemma 2.1 in \cite{F}}

The mistake in the proof of Lemma 2.1 in \cite{F} relies in the confusion
of the matrix $\t$ with entries $(\t)_{ij}=t_i ^j$ between the inverse
of $\t$ and the inverse of the transposed matrix of $\t$. Even though I 
do not have a concrete counter-example, I think Lemma 2.1 is false in its
 full generality. However, one can still view Dubois-Violette and
 Launer's Hopf algebra as a universal bialgebra construction. 
Recall briefly the universal construction in \cite{F}:

\begin{defi} Let $V$ be a finite dimensional vector space
with basis $\{x_i\}_{i=1}^{\dim V}$ and $f:V ^{\ot n_1}\to
V^{\ot n_2}$ be a linear map. Consider free generators
$t_i ^j$ ($i,j=1,\dots,\dim V$) and using multi-index notation
\[
x_I:=x_{i_1}\ot x_{i_2}\ot\cdots \ot x_{i_\ell}\in V^{\ot \ell}
\]
\[
t_I^J:=t_{i_1}^{j_1}t_{i_2}^{j_2}\cdots t_{i_\ell}^{j_\ell}\in k\{t_i ^j,
i,j=1,\dots,\dim V\}
\]
write $f(x_I)=\sum_{J}f_I^Jx_J$ and define
 the two-sided ideal $\II_ f:=\big\langle
\sum_{J}(t_I^Jf_J^K-f_I^Jt_J^K ) : \ \forall I,K\big\rangle$
and the algebra
\[
A(f):=k\{t_i ^j:i,j=1\dots,\dim V\}/\II_f
\]
with comultiplication induced by
\[
\Delta t_i ^j=\sum_{i=1}^{\dim V}t_i ^k\ot t_k^j
\]
If $F=\{f_i:V ^{\ot n^i_1}\to V ^{\ot n^i_2}:i\in I\}$ is a family
of linear maps indexed by a set $I$, define $\II_F:=\sum_{i\in I}
\II_{f_i}$ and $A(F):=k\{t_i ^j:i,j=1\dots,\dim V\}/\II_F$
\end{defi}

\subsection{Dubois-Violette and Launer's Hopf algebra as a universal bialgebra}

If $b:V\times V\to k$ is a non-degenerate
bilinear form and using the notation
$b_{ij}=b(x_i,x_j)$, we consider {\em two} linear maps
\[
b:V^{\ot 2}\to V=V ^{\ot 0}\]
\[
x_i\ot x_j\mapsto b_{ij}
\]
and
\[
i_b:k\to V^{\ot 2}
\]
\[
1\mapsto 
\sum_{i,j}b^{ij}x_i\ot x_j
\]
where $b^{ij}$ are the $ij$-entries of the
inverse of the matrix $(B)_{ij}=b_{ij}$.

If we denote $H_{DV-L}(b)$ the Dubois-Violette and Launer's Hopf
 algebra, one tautologically has that 
\[
H_{DV-L}=A(b,i_b)
\]
(but not $A( b)$).

For reasons that will be clear soon, let us write $t_{\mu\nu}:=t_{\mu} ^{\nu}$.
Let us denote $B$ the matrix with indices
$(B)_{\mu\nu}=b_{\mu\nu}$, and keep the ''up convention'' for 
$b ^{\mu\nu}=(B ^{-1})_{\mu\nu}$.
Then the above equations are (sum over repeated indexes)
\begin{eqnarray}
b_{\mu\nu}t_ {\lambda\mu} t_{\rho\nu} = b_{\lambda\rho}1
\\
b^{\mu\nu}t_{\mu \lambda}t_{\nu\rho} = b^{\lambda\rho}1
\end{eqnarray}
We see that if $\t$ is the matrix with entries $(\t)_{ij}=t_{ij}$ then
the equations are
\begin{eqnarray}\label{m1}
\t B\t ^{tr}=B\
\\
\label{m2}
\t ^{tr}B^{-1}\t=B ^{-1}
\end{eqnarray}
where $\t ^{tr}$ denotes the transposed matrix.

Equation \eqref{m1} says that $\t$ has a right inverse (and $\t ^{tr}$
has a left inverse), but it is not obvious that this single equation implies that
 $\t$ has a left inverse (or that  $\t ^{tr}$ has a right inverse).
But clearly equation \eqref{m2} says that $\t$ has an inverse from the 
other side. The key point when proving both axioms of the antipode is
to prove that a given matrix has both left and right inverse. We
formalize the statement in the following lemma:

\begin{lem}
Assume $H$ is a bialgebra generated by some group-like elements and
a set $\{t_{ij},i,j=1,\dots,n\}$ with 
$\Delta(t_{ij})=\sum_{k=1} ^nt_{ik}\ot t_{kj}$
and $\epsilon(t_{ij})=\delta_{ij}$.
Denote  $\t\in M_n(H)$ the $n\times n$ matrix with entries
$(\t)_{ij}$. 
 Then the bialgebra $H$ is a Hopf algebra
if and only there exist an anti-algebra morphism $S:H\to H$ such
that $S(D)=D ^{-1}$ for all group-like generators $D$ and 
the matrix  $S(\t)$ with coefficients $(S(\t))_{ij}=S(t_{ij})$ is
the inverse of the matrix $\t$ in $M_n(H)$.
\end{lem}

\begin{proof}
Assume $H$ is a Hopf algebra. The antipode
axiom says in particular
\[
m(\Id\ot S)\Delta(t_{ij})=\epsilon(t_{ij})=m(S\ot \Id)\Delta(t_{ij})
\]
But because of the comultiplication and counit properties of the $t_{ij}$ these equations translates into
\[
\sum_{k=1}^nt_{ik} S(t_{kj})=
\delta_{ij}=\sum_{k=1}^nS(t_{ik}) t_{kj}
\]
Denoting $S(\t)$ the matrix with entries $S(\t)_{ij}=S(t_{ij})$,
  the above equation forall $ij$ is simply the entries of the 
single matrix equation
\[
\t\cdot S(\t)=\Id_{n\times n}=S(\t)\cdot\t
\]
On the other hand, assume $H$ is a bialgebra generated by group-like 
elements and a set $\{t_{ij}:i,j=1,\dots,n\}$ where the elements
$t_{ij}$ satisfy 
$\Delta(t_{ij})=\sum_{k=1} ^nt_{ik}\ot t_{kj}$
and $\epsilon(t_{ij})=\delta_{ij}$. If $S(\t)=\t ^{-1}$ and
 $S(D)=D ^{-1}$ for
all group-like in the set of generators, then clearly $S$ satisfies
the antipode axiom on generators. Hence, $S$ is the antipode for $H$
and $H$ is a Hopf algebra.
\end{proof}

Recall the notation in $\cite{F}$: $c:V^{\ot 2}\to V^{\ot 2}$
is a solution of the braid equation and $A(c)$ is its universal bialgebra,
that is the algebra with generators $t_{ij}$ and relations
\[
\sum_{k,\ell}c_{ij}^{k\ell}t_{kr}t_{ \ell s}=
\sum_{k,\ell}t_{ik}t_{ j\ell} c_{k\ell}^{rs}
\hskip 1cm \forall\ 1\leq i,j,r,s\leq n.
\]
that coincides with the FRT construction \cite{FRT}.
Using the same idea as in the above lemma
 we have:

\begin{lem}\label{mainlemma}
Assume $D\in A(c)$ is a group-like element such that
the matrix $\t\in M_n\big(A(c)[D ^{-1}]\big)$ 
is invertible, then
$A(c)[D ^{-1}]$ is a Hopf algebra.
\end{lem}
\begin{proof}
Assume $\t$ is an invertible matrix, call  $\u$ its inverse 
and $u_{ij}:=(\u)_{ij}$. Let us prove that there exists a unique
well-defined anti-algebra map
\[
A(c)\to A(c)[D ^{-1}]
\]
\[
t_{ij}\mapsto u_{ij}
\]
Since $A(c)$ is freely generated by the $t_{ij}$ with relations
\begin{equation}\label{FRTeq}
\sum_{k,\ell}c_{ij}^{k\ell}t_{kr}t_{ \ell s}=
\sum_{k,\ell}t_{ik}t_{ j\ell} c_{k\ell}^{rs}
\hskip 1cm \forall\ 1\leq i,j,r,s\leq n.
\end{equation}
one should check the opposite relation in $A(c)[D ^{-1}]$
\[
\sum_{k,\ell}c_{ij}^{k\ell}u_{ \ell s}u_{kr}\overset{?}{=}
\sum_{k,\ell}u_{ j\ell}u_{ik} c_{k\ell}^{rs}
\hskip 1cm \forall\ 1\leq i,j,r,s\leq n.
\]
But because the matrix  $\t$ is invertible in $M_n(A[D^{-1}])$,
we apply the operator
\[
\sum_{r,s,i,j}t_{di}t_{cj}\big(-\big)t_{ra}t_{sb}
\]
and we get the equivalent checking
\[
\sum_{k,\ell,r,s,i,j}t_{di}t_{cj}c_{ij}^{k\ell}u_{ \ell s}u_{kr}t_{ra}t_{sb}\overset{?}{=}
\sum_{k,\ell,r,s,i,j}t_{di}t_{cj}u_{ j\ell}u_{ik} c_{k\ell}^{rs}t_{ra}t_{sb}
\hskip 1cm \forall\ 1\leq i,j,r,s\leq n.
\]
Now using 
(on LHS) $\sum_{r,s}u_{ \ell s}u_{kr}t_{ra}t_{sb}=\delta_{\ell b}\delta_{ka}$ and (on RHS)
 $\sum_{ij}t_{di}t_{cj}u_{ j\ell}u_{ik}=\delta_{dk}\delta_{c\ell} $
we get
\[
\sum_{i,j}t_{di}t_{cj}c_{ij}^{ab}
\overset{?}{=}
\sum_{r,s} c_{dc}^{rs}t_{ra}t_{sb}
\]
and this is the same relation as \eqref{FRTeq}, that is valid on
$A(c)$, hence, it is valid in $A(c)[D ^{-1}]$ as well.

Recall that the non-commutative localization
$A(c)[D ^{-1}]$ is the algebra freely generated by $A(c)$ and the symbol 
$D ^{-1}$ with (the same relations as in $A(c)$ and) 
\[
DD^{-1}=1=D ^{-1}D
\]
Having defined an antialgebra map $A(c)\to A(c)[D ^{-1}]$ we extend to
a map $A(c)[D ^{-1}]\to A(c)[D ^{-1}]$ by sending $D\mapsto D ^{-1}$
and this define the desired map $S$: the antipode axioms for $S$
are easily 
checked
on generators.
\end{proof}

\section{Nichols Algebras and an alternative to Theorem 3.10 of \cite{F}}

We will use the following well-known
 facts from finite dimensional Nichols algebras.
Assume $(V,c)$ a rigid solution of YBeq such that  $\B(V,c)$ is finite
 dimensional.
 \begin{fact} $\B$ is a graded algebra and coalgebra.
It is not a Hopf algebra in the usual sense, but it is a Hopf algebra
in the category of Yetter-Drinfeld modules over some Hopf algebra $H$.
\end{fact} 
For instance, $H=H(c)$ the Hopf envelope of  $A(c)$ do the work.

 \begin{fact} Denoting $\B^{top}$ the highest non-zero degree of $\B$,
one has $\dim\B^{top}=1$, say $\B^{top}=k \b$ for
a choice of a non-zero element $\b\in\B^{top}$. 
The projection into the coefficient
of $\b$
\[
\B\ni\om=\om_0+\om_1+\cdots+\om_{top}=
\om_0+\om_1+\cdots+\lambda \b
\ \ \longmapsto \ \ \lambda\in k\]
is an integral of the braided Hopf algebra $\B$. Also,
because of $\dim\B^{top}=1$, 
the $A(c)$-comodule structure gives a non trivial grouplike
element $D$ determined by
\[
\rho(\b)=D\ot \b
\]

\end{fact}

\begin{fact}  \label{ev}
For each degree, the multiplication map
induces a  non-degenerate pairing
\[
m|:\B ^p\ot\B^{top-p}\to\B ^{top}
\]
 In particular, since $\B ^1=V$, the restriction of the multiplication
\[
V\ot\B ^{top-1}\to \B ^{top}=\b k\]
gives a non-degenerate pairing. In terms of basis, if $\om ^i$ is a basis
of $\B ^{top-1}$ and $x_i$ is a basis of $V$, then
write
\[
\om ^ix_j=m_{ij}\b,\ \ m_{ij}\in k
\]
and the matrix $(m_{ij})$ is invertible, and one can choose a
''dual basis'' $\om ^i$ such that
\[
x_i\om ^j=\delta_{ij}\b
\]
\end{fact} 

\begin{fact} If $(V,c)$ is rigid then so is $(V ^*,c ^*)$, and
$\B(V ^*,c ^*)$ is the graded dual (as algebra and coalgebra)
of $\B(V,c)$.
\end{fact} 

\begin{fact} \label{coev} By duality (using Fact \ref{ev}),
denoting $\coev$ the comultiplication composed with projection
\[
\xymatrix{
\B ^{top}\ar@/_3ex/[rr]_{\coev}\ar[r] ^<<<<<{\Delta}
&
\underset{p}{\oplus}\B ^{top-p}\ot \B^p\ar[r] ^{\pi}&\B^{top-1}\ot V},
\]
it is non degenerate, in the sense that if $x_i$, $\wh \om ^i$ are bases
of $V$ and $\B ^{top-1}$ respectively, and
\[
\coev(\b)=\sum_{ij}\coev_{ij}\wh\om^i \ot  x_j,\ \ \ \coev_{ij}\in k
\]
then the matrix $(\coev_{ij})$ is invertible. In particular,
there exists a basis $\wh\om ^i$ such that
\[
\coev(\b)=\sum_i \wh\om^i \ot  x_i
\]
\end{fact} 

\begin{fact} The maps in \ref{ev} and \ref{coev} are
$A(c)$-colinear.
\end{fact}

Now denote $\rho:\B\to A(c)\ot \B$ the comodule structure map
and write  $\rho(\om ^j)=T_ {jk}\ot \om ^k$
and $\rho(\wh \om ^j)=\wh T_ {jk}\ot \wh\om ^k$, where 
$\{\om ^j\}$ and $\{\wh\om^j\}$ are basis of $\B ^{top-1}$ as
in \ref{ev} and \ref{coev}.
By $A(c)$-colinearity we have
\[
D\ot\delta_{ij}\b
=\rho(\delta_{ij}\b)=
\rho(x_i\om ^j)
=
t_{ ik}T_ {jl}\ot x_k\om ^l
=
t_{ ik}T_ {jl}\ot \delta_k ^l\b=
t_{ ik}T_ {jk}\ot \b
\]
\[
\To
t_{ ik}T_ {jk}=D\delta_{ij}\]
\[
\iff
\t\cdot  T ^{tr}=D\id\]

Now using the $\wh\om ^j$'s:
\[
\rho(\sum_i \wh\om^i\ot x_i)=
\rho(\coev( \b))
=(1\ot \coev )\rho(\b)\]
\[
=(1\ot \coev )(D\ot\b)
=
D\ot \sum_i (\wh\om^i\ot x_i)
\]
but also

\[
\rho(\sum_i \wh\om^i\ot x_i)=
\sum_{i,j,k}\wh T_ {ij} t_{ik}\ot \wh\om^j\ot x_k
\]
This proves
\[
\wh T_ {ij} t_{ik}=D\delta_{jk}
\To
\wh T ^{tr}\cdot \t=D\id
\]
hence, $D ^{-1}\wt T ^{tr}$ is a left inverse of $\t$, that is,
$\t$ is invertible in $M_n\big(A(c)[D ^{-1}]\big)$.

Using \ref{mainlemma} and observing that $A(c)=A(qc)$ 
for any $0\neq q\in k$, one can conclude the following main result,
 that is an alternative to Theorem 3.10 in \cite{F}:

\begin{teo}\label{mainthm}
Let $V$ be a finite dimensional vector space, $c:V ^{\ot 2}\to V ^{\ot 2}$
a rigid solution of the braid equation and assume there is a non-zero scalar
$0\neq q\in k$ such that $\B:=\B(V,cq)$ is finite dimensional. Denote
$D$ the associated group-like element in $A(c)$ coming from
$\B^{top}$. Then $A(c)[D^{-1}]$ is a Hopf algebra.
\end{teo}


\begin{thebibliography}{99}



\bibitem[DV-L]
{DVL}
{\sc M. Dubois-Violette and G. Launer},
{\em The quantum group of a non-degenerate bilinear form},
Physics Letters B,
Volume 245, number 2 (1990) pp.  175-177.

\bibitem[FRT]
{FRT}
{\sc L.D. Faddeev, N.Yu. Reshetikhin} and {\sc L.A. Takhtajan},
{\em Quantization of Lie groups and
Lie algebras}, Leningrad Math. J. 1 (1990) 193.

\bibitem[F]{F} {\sc M. Farinati},
{\em Universal quantum (semi)groups and Hopf envelopes}.
Algebra and Rep. Theory (2022).
\url{https://doi.org/10.1007/s10468-022-10122-9}.


\bibitem[FG]
{FG}
{\sc M . Farinati and G. A. Garc\'ia},
 {\em
 Quantum function algebras from finite-dimensional Nichols algebras},
  J. Noncommutative Geometry,
\textbf{14}(3), (2020), 879--911.
\url{https://doi.org/10.4171/jncg/381}.




\end{thebibliography}
\end{document}